\documentclass[a4paper]{article}

\usepackage[english]{babel}
\usepackage[T1]{fontenc}

\usepackage[a4paper,top=3cm,bottom=2cm,left=3cm,right=3cm,marginparwidth=1.75cm]{geometry}

\usepackage{nicefrac}
\usepackage{paracol}
\usepackage{caption}
\usepackage{subcaption}
\usepackage{amsmath}
\usepackage{amssymb}
\usepackage{graphicx}
\usepackage[colorinlistoftodos]{todonotes}
\usepackage[colorlinks=true, allcolors=blue]{hyperref}
\usepackage{amsthm}
\usepackage{amsfonts}

\RequirePackage[capitalize,nameinlink]{cleveref}[0.19]

\crefname{section}{section}{sections}
\crefname{subsection}{subsection}{subsections}
\Crefname{section}{Section}{Sections}
\Crefname{subsection}{Subsection}{Subsections}
\Crefname{figure}{Figure}{Figures}
\crefformat{equation}{\textup{#2(#1)#3}}
\crefrangeformat{equation}{\textup{#3(#1)#4--#5(#2)#6}}
\crefmultiformat{equation}{\textup{#2(#1)#3}}{ and \textup{#2(#1)#3}}
{, \textup{#2(#1)#3}}{, and \textup{#2(#1)#3}}
\crefrangemultiformat{equation}{\textup{#3(#1)#4--#5(#2)#6}}%
{ and \textup{#3(#1)#4--#5(#2)#6}}{, \textup{#3(#1)#4--#5(#2)#6}}{, and \textup{#3(#1)#4--#5(#2)#6}}
\Crefformat{equation}{#2Equation~\textup{(#1)}#3}
\Crefrangeformat{equation}{Equations~\textup{#3(#1)#4--#5(#2)#6}}
\Crefmultiformat{equation}{Equations~\textup{#2(#1)#3}}{ and \textup{#2(#1)#3}}
{, \textup{#2(#1)#3}}{, and \textup{#2(#1)#3}}
\Crefrangemultiformat{equation}{Equations~\textup{#3(#1)#4--#5(#2)#6}}%
{ and \textup{#3(#1)#4--#5(#2)#6}}{, \textup{#3(#1)#4--#5(#2)#6}}{, and \textup{#3(#1)#4--#5(#2)#6}}
\crefdefaultlabelformat{#2\textup{#1}#3}

\newtheorem{theorem}{Theorem}[section]
\newtheorem{definition}[theorem]{Definition}
\AddToHook{env/definition/begin}{\crefalias{theorem}{definition}}
\crefname{definition}{Definition}{Definitions}

\newtheorem{corollary}[theorem]{Corollary}
\AddToHook{env/corollary/begin}{\crefalias{theorem}{corollary}}
\crefname{corollary}{Corollary}{Corollaries}

\newtheorem{proposition}[theorem]{Proposition}
\AddToHook{env/proposition/begin}{\crefalias{theorem}{proposition}}
\crefname{proposition}{Proposition}{Propositions}
\newtheorem{observation}[theorem]{Observation}
\AddToHook{env/observation/begin}{\crefalias{theorem}{observation}}
\crefname{observation}{Observation}{Observations}
\newtheorem{hypothesis}[theorem]{Hypothesis}
\AddToHook{env/hypothesis/begin}{\crefalias{theorem}{hypothesis}}
\crefname{hypothesis}{Hypothesis}{Hypotheses}

\theoremstyle{definition}
\newtheorem{remark}[theorem]{Remark}

\newcommand{\email}[1]{\protect\href{mailto:#1}{#1}}
\usepackage{bbm}
\usepackage{algorithm}
\usepackage{algorithmic}
\numberwithin{algorithm}{section}

\usepackage{enumerate}
\usepackage{mathtools}
\usepackage[shortlabels]{enumitem}
\usepackage{microtype}

\usepackage{tikz}
\usepackage{tikz-cd}
\usetikzlibrary{shapes, arrows, positioning}
\usetikzlibrary{matrix, calc, arrows}

\title{Learning iterated function systems from time series of partial observations}
\author{Emilia Gibson\footnote{Department of Mathematics, Imperial College London, London SW7 2AZ, UK
  (\email{emilia.gibson19@imperial.ac.uk} , \email{jeroen.lamb@imperial.ac.uk}).}
   \ and Jeroen S.W. Lamb\footnotemark[1]\text{ }\footnote{International Research Center for Neurointelligence, The University of Tokyo, Tokyo, 113-0033 Japan.}\text{ }\footnote{Centre for Applied Mathematics and Bioinformatics, Department of Mathematics and Natural Sciences, Gulf University for Science and Technology, Halwally, 32093 Kuwait.}
}

\begin{document}
\maketitle

\begin{abstract}
We develop a methodology to learn finitely generated random iterated function systems from time-series of partial observations using delay embeddings.
We obtain a minimal model representation for the observed dynamics, using a hidden variable representation, that is diffeomorphic to the original system.
\end{abstract}

\section{Introduction}
In data-driven modelling of dynamical systems, learning equations of motion from time-series observations is a central objective and many methodologies have been developed to learn ordinary and partial differential equations (and their discrete time analogues), as well as transfer and Koopman operators in these settings.

In recent decades, there has been a steadily growing interest in dynamical systems whose equations of motion contain a probabilistic element, so-called \textit{random dynamical systems} (RDSs). 
These kind of dynamical systems represent dynamical systems driven by ``noise'', reductions of (higher dimensional) complex dynamical systems where ignored degrees of freedom are represented in a probabilistic way, or express models with an inherent uncertainty.  
Models with noise are state-of-the-art in topical applications in the life sciences, physical sciences, social sciences and finance
\cite{kobayashi2011probability}.
Whereas admittedly the most common modelling framework for random dynamical systems is stochastic (partial) differential equations (S(P)DEs), there are many other types of random dynamical systems (that are  more suitable than S(P)DEs in certain contexts). 

In this paper we address the task of learning one of the most elementary types of random dynamical systems: finitely generated random iterated function systems. 
A random iterated function system refers to a discrete-time dynamical system, whose evolution consists of the sequential application of self-maps 
(functions) on a (metric) state space, where the selection process of the sequence of maps is governed by probability. 
We assume that the number of functions is finite and the probabilistic process determining the sequence(s) is that of an irreducible finite state Markov Chain (MC), 
with the states of this MC representing the generating functions.  

Let $\Sigma_{P}^+$ denote an irreducible finite ($k$) state Markov chain with $k\times k$ transition matrix $P$, and
let the sequence $\omega:=\omega_0\omega_1\omega_2\ldots$ denote a sample. 
Then, we consider the time-evolution of the random iterated function system generated by a finite set of functions with compact domain $M$
\begin{equation}\label{eq: ifs maps}
f_i:M\to M, ~i=1,\ldots, k.
\end{equation}
Given a sample $\omega\in\Sigma^+_P$ and an initial condition $x_0\in M$, the evolution of the system is given by  
\begin{equation}\label{ifs}
x_{n+1} = f_{\omega_n}(x_n),~n=0,1,2,\ldots.
\end{equation}
The transition matrix $P$ governs the probability of the occurrence of different types of sequences $\omega$.
We refer to a random dynamical system of the above type as a \textit{finitely generated random Iterated Function System}, and in abbreviation as IFS.\footnote{IFSs have been studied extensively in the case that all the maps $f_i$ are contractions, see e.g. \cite{barnsley2012fractals}, but in a broader RDS setting (as in this paper), we do not require any properties of the functions $f_i$ other than some mild smoothness and non-degeneracy conditions.}
The functions $f_1,\dots,f_k$ are thus referred to as the generators of the IFS. 

We consider the setting in which the trajectories of the IFS $x_0x_1x_2x_3\ldots$ are observed through a smooth observable $\psi:M\to \mathbb{R}^s$, and the task of learning the ground truth IFS from observations $z_0z_1z_2z_3\ldots$, where $z_i:=\psi(x_i)$.
Since from observations through the lens of $\psi$, we have inherently no knowledge about the coordinate representation of the IFS on $M$, the task of \textit{learning} at best produces an IFS that is diffeomorphic to the the ground truth IFS and accurately (re)produces the time-series of observations. 
Of course, the extent to which one is able to learn the system also intrinsically depends on the amount of time-series data we have available, but we will not dwell on estimating the amount of data needed to obtain an approximation with given tolerance, so our abstract results refer to the idealisation in which we have unlimited data.

Our main result is summarized as follows:
\begin{theorem}\label{th:main}
Under mild smoothness and non-degeneracy conditions,  a finitely generated random iterated function system can be learnt from time-series of partial observations. 
\end{theorem}
This result is constructive and builds on delay-embeddings for iterated function systems \cite{stark}, \cite{IFSchannels}, hidden variable inference \cite{stepaniants2024discovering} and the identifiability of delay-embedded finite state MCs.
The precise conditions are given in \cref{hyp: embedding}. 

\begin{figure}
    \centering
    \includegraphics[width=\textwidth, trim = 1cm 16cm 1cm 0cm]{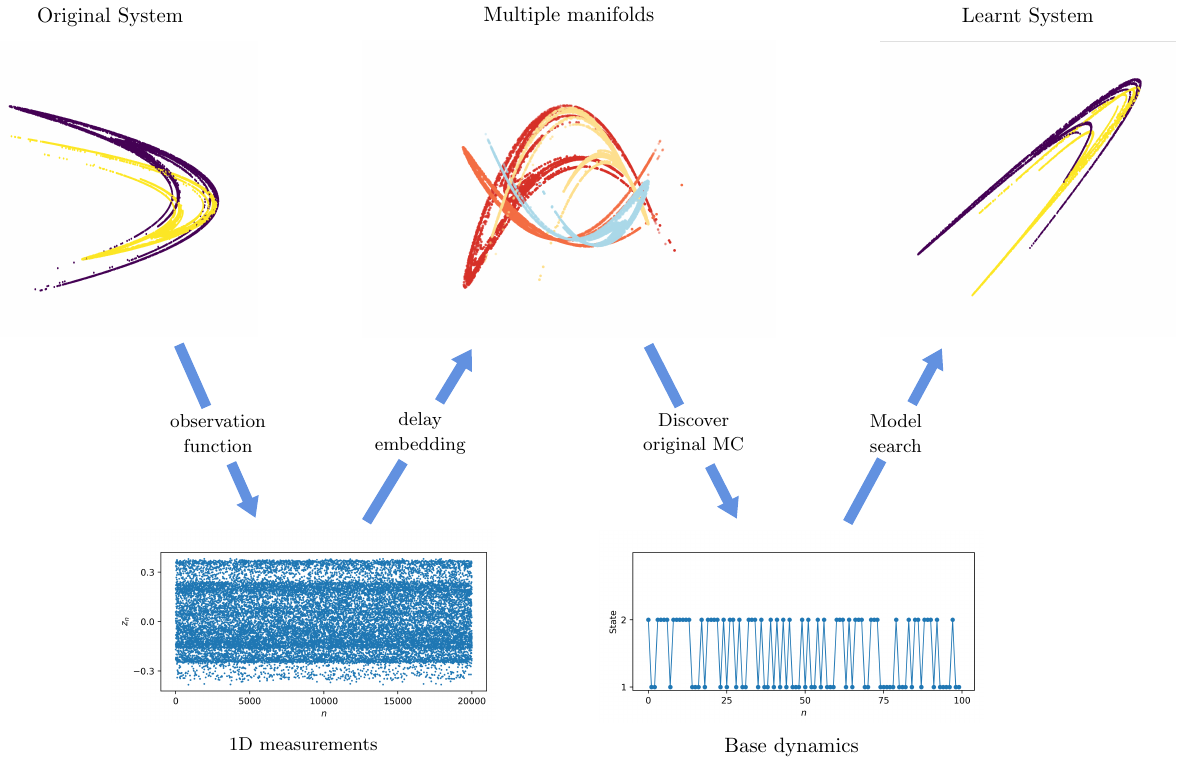}
    \caption{Flow diagram sketch of our approach towards learning IFSs, with reference to the H\'enon IFS defined in \cref{subsection: Henon IFS}.
    Points identified as iterates under different generators are shown in different colours.
    \textbf{Top, left}: a sample trajectory from the underlying system, which here consists of two generating functions.
    \textbf{Bottom, left}: scalar observation data.
    \textbf{Top, middle}: the 1D measurements delay embedded in $\mathbb{R}^3$. In the delay space there are multiple manifolds appearing, here four, which we separate. 
    Using transitions between different manifolds, we discover the sequence of generators from the sampled trajectory.
    \textbf{Bottom, right}: The discovered sequence of generators used in the sample trajectory.
    We fit an IFS model to the observed time series and its discovered sequence of generators.
    \textbf{Top, right}: A sample trajectory from the learnt IFS model.
    }
    \label{fig: schematic}
\end{figure}
We sketch our approach  in \cref{fig: schematic}, with some illustrations of data from an example that we discuss later on in this paper (H\'enon IFS, \cref{subsection: Henon IFS}). 
Our methodology consists of three stages:
\begin{itemize}
\item Finding a delay embedding of the observations so as to reveal the existence of a finite number of generators of the delay-embedded IFS, with associated transition probabilities.
\item Reconstruct the MC structure of the ground truth IFS from the MC structure of the delay-embedded IFS.
\item Construct generators for an IFS that is bundle diffeomorphic to the ground truth IFS using hidden variable representations.
\end{itemize}
We develop the methodology step by step in the following sections: \cref{section: multi-manifold clustering,section: TDEMCS,section: identifying conjugate mappings}, with each section dedicated to one of the three stages listed above.
In these sections, we discuss the feasibility of each of these steps, as well as the theory that underlies \cref{th:main}.
Taken together, \cref{section: multi-manifold clustering,section: TDEMCS,section: identifying conjugate mappings} provide the foundation for \cref{th:main}, for further details see \cref{proof: main}.
In \cref{section: numerics} we provide some numerical examples to demonstrate the potential of our approach, and follow up with conclusions in \cref{section: conclusions}.


\section{Detection and separation of generators}\label{section: multi-manifold clustering}

We first discuss how an IFS presents itself in time series data. 
This question was previously studied in \cite{detectingIFS} for the case of full observations ($\psi=Id$). 
We propose an alternative framework involving multi-manifold learning, that is well-suited for general observation functions.

\subsection{IFSs and multi-manifold learning}\label{subsection: MMC}

Consider first the setting of full observations, i.e., $\psi = Id$.
Given the existence of $k$ different generators, a natural first step is to divide the observations $\{x_n\}_{n\ge0}$ into $k$ clusters, so that each point is identified as an iterate under one of the different generators. 
An intuitive approach is to reference the graphs of the different functions in the IFS.
We assume the following.

\begin{hypothesis}\label{hyp: transversal}
We assume that $M$ is a $r$-dimensional manifold and that the set of generators in \cref{eq: ifs maps} are $C^1$ smooth and transversal.
\end{hypothesis}

We map time series data to the graphs of the different functions by introducing the lag-1 vectors 
$(x_n,x_{n+1})$.
Then clustering points in $M\times M$ which belong to a common manifold together (see literature on multi-manifold clustering for possible algorithms, e.g. \cite{wang2010multi}) and using relation \cref{ifs}, we recover the sample $\omega$ which generated the data.
Expressions for each of the functions in the IFS are found using standard function fitting techniques.
This straightforward approach can be very effective.
Multi-manifold clustering methods work best when data points are evenly distributed across each of the different manifolds.
Note that prior knowledge of $k$ is not essential.

\cref{hyp: transversal} ensures that each graph is a smooth topological manifold in the product space, and that these graphs are far apart in the Whitney $C^1$ topology and can hence be separated.
Instead of requiring that the set of generators to be transversal, it suffices to assume that the different generators are separable in the Whitney $C^q$ topology for some $q\ge 1$ finite. 

\subsubsection{Family of logistic maps}
Consider the 1D IFS consisting of 3 logistic maps:
\begin{equation*}f_1(x) = 3 x(1-x), \quad f_2(x) = 3.5 x(1-x), \quad f_3(x) = 4 x(1-x),\end{equation*}
where each map has an equal probability of being chosen at any given time step.
The lagged vectors $(x_n,x_{n+1}) \in [0,1]^2$ of a sample trajectory are displayed in \cref{fig: Lift in space.}.
Points are confined to the three graphs of the different functions. 
Using the different cluster, we then 
recover the sequence of generators $\omega$ of the time series.
\begin{figure}
    \centering
    \includegraphics[width=0.5\textwidth]{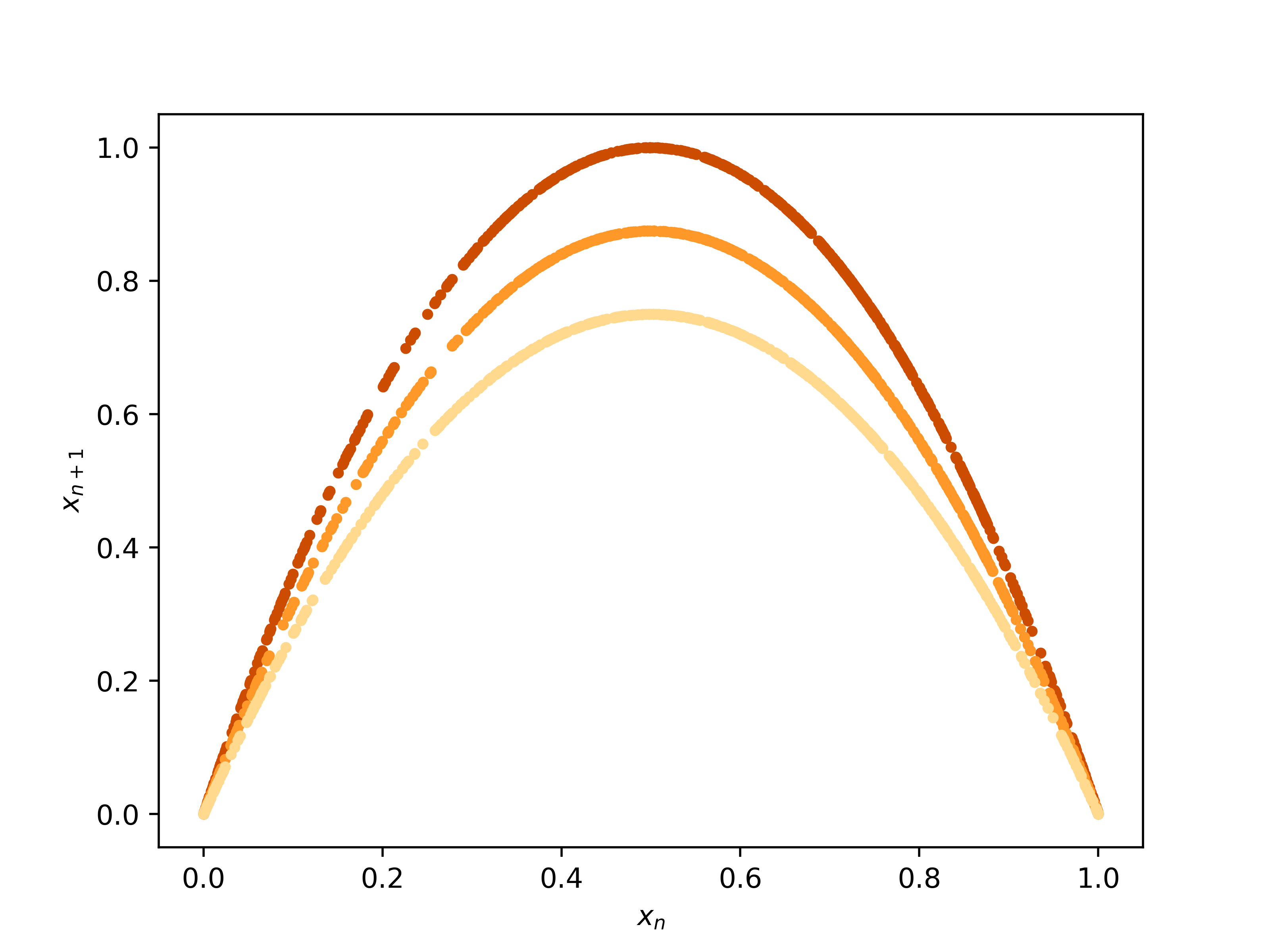}
    \caption{(Family of logistic maps). A lag-1 plot of a sample trajectory consisting of 2200 steps. 
    The lag-1 vectors lie on a collection of manifolds, representing the graphs of the functions in the IFS, shown in different colours.}
    \label{fig: Lift in space.}
\end{figure}

\subsection{Delay embeddings of IFSs}\label{subsection: delay embeddings}
In the previous subsection, we assumed that we could measure the entire state space. 
This is a strict requirement, and in practical settings, is often an unrealistic assumption.
We relax this condition and assume instead that we observe the system through some measurement function $\psi:M\to \mathbb{R}^s$.
We think of these `partial observations' as a mapping to a lower dimensional space.

In the deterministic setting it is often assumed that the ground truth can be reconstructed from partial observations. 
A standard approach for reconstructing the ground truth system is to use time-delay embeddings \cite{PackardCrutchfield,takens}.
In previous work by J. Stark and D. S. Broomhead, it was shown that a ground truth IFS can also be reconstructed using time-delay embeddings.
Our contribution is to use the reconstructed system $\mathbf{z}_n(l):= (z_{n},z_{n+1},\dots,z_{n+l-1})$ to find an analytic model for the original system.
To simplify notation, we formulate the results for continuous scalar observables $\psi:M\to\mathbb{R}$, but these results readily extended to higher dimensional observations \cite[Remark 2.9]{sauer1991embedology}.

Given an IFS \cref{ifs}, the delay vector $\mathbf{z}_n(l)$ depends on the observation $x_n$ as follows:
\begin{equation}
\mathbf{z}_n(l) = \left(\psi(x_n), \psi ( f_{\omega_{n}}(x_n)), \dots, \psi(f_{\omega_{n+l-2}}\circ\dots\circ f_{\omega_{n}}(x_n))\right).
\end{equation}
Hence, $\mathbf{z}_n(l)$ depends on $\omega$ through $\Omega_n:=(\omega_n,\omega_{n+1},\dots,\omega_{n+l-2})$. 
Let $\Phi_{\psi,f,\Omega_n}:M\to\mathbb{R}^l$ be the delay coordinate map given by
\begin{equation}\label{eq: delay embedding map}
    \Phi_{\psi,f,\Omega_n}(x_n) = \left(\psi(x_n), \psi ( f_{\omega_{n}}(x_n)), \dots, \psi(f_{\omega_{n+l-2}}\circ\dots\circ f_{\omega_{n}}(x_n))\right),
\end{equation}
then we write each delay vector in terms of its delay map as $\mathbf{z}_n(l) = \Phi_{\psi,f,\Omega_n}(x_n)$.
A generic IFS delay embedded in $\mathbb{R}^{l}$ admits up to $k^{l-1}$ different delay maps, and each map has an associated $(l-1)$-tuple $\underline{i}\in\Sigma_P^{l-1}$. 

Let $M$ be a smooth compact $r$-dimensional manifold. 
It follows from the embedding result in \cite{stark}, which we state below for completeness, that if $l\ge2d+1$, then each delay map $\Phi_{f,\psi,\underline{i}}$ is generically an embedding.
\begin{theorem}[Takens' theorem for IFSs \cite{stark}]\label{thm: stark}
    Let $M$ be a $1\le r$-dimensional smooth compact manifold. 
    Let $C^1(M,\mathbb{R})$ be the space of $C^1$ real-valued functions on $M$ and $D(M)$ be the space of diffeomorphisms of $M$.
    Suppose $l\ge2r+1$. Then there exists an open and dense set of $(\psi,f) \in C^1(M)\times\left[D(M)\right]^k$ such that, for any $(\psi,f)$ in this set, $\Phi_{\psi,f,\underline{i}}$ is an embedding for every $\underline{i}\in\Sigma_P^{l-1}$.
\end{theorem}

The dynamics in the delay space takes the form
\begin{equation}\label{eq: delay dynamics}\mathbf{z}_{n+1}(l) = \Phi_{\psi,f,\Omega_{n+1}} \circ f_{\omega_n}\circ\Phi_{\psi,f,\Omega_n}^{-1}(\mathbf{z}_n(l)) =: F_{\omega_n,\dots,\omega_{n+l-1}}(\mathbf{z}_n(l)),
\end{equation} 
where the subscript $(\omega_n,\dots,\omega_{n+l-1})$ refers to the fact that one time step in the delay space depends on  the $l$ most recent entries of $\omega$. 
By \cref{thm: stark}, the ground truth system and the reconstructed dynamics in the delay space 
are bundle diffeomorphic. Here, the term \textit{bundle diffeomorphism} refers to the fact that the smooth coordinate transformation $\Phi$ in \cref{eq: delay dynamics} is dependent on $\omega$.

By learning closed form expressions for each of the generators of the delay embedded IFS, we can use the dynamics in \cref{eq: delay dynamics} to model the ground truth IFS.
However, as the original system is non-autonomous, the symbolic dynamics of the delay embedded system also have a delay embedded structure.
Therefore, model representations of the delay embedded system have an increased complexity that is not reflected in the original dynamics.
This leads to the following question. 
From the delay embedded system, can we find a parsimonious representation in which the ground truth IFS and the inferred IFS have the same number of generators?

To this end, we first aim to discover the associated sample(s) $\omega$ that generated the observed time series. 
We begin by identifying each delay vector with its corresponding delay map.
We assume the following. 
\begin{hypothesis}[Generic IFS]\label{hyp: embedding}\hfill
    \begin{enumerate}
        \item The generators of the IFS are diffeomorphisms of $M$, 
        \item the observable $\psi:M\to \mathbb{R}$ is continuous,
        \item there exist integers $l,q\ge 1$ finite, and $\epsilon>0$ such that the delay maps $\Phi_{\psi,f,\underline{i}}$ are embeddings for every $\underline{i}\in\Sigma_P^{l-1}$, and the pairwise distance between delay maps is at least $\epsilon$ in the Whitney $C^q$ topology for all pairs excluding the diagonal.
    \end{enumerate}
\end{hypothesis}

Since each delay map is an embedding, the images of $M$ under different delay map correspond to different $r$-dimensional sub-manifolds in the delay space \cite{IFSchannels}, and so the delay vectors $\mathbf{z}_n(l)$ belong to the following collection of manifolds:
\begin{equation*} \bigcup_{\underline{i}\in\Sigma_P^{l-1}} \Phi_{\psi,f,\underline{i}\hskip0.1em}(M), \quad \forall n\ge0.
\end{equation*}
Therefore, we propose to use multi-manifold learning to cluster together delay vectors that belong to the same delay map.
While delay vectors are assigned to different manifolds, we still need resolve their associated $(l-1)$-tuples.
This problem is the focus of \cref{section: TDEMCS}.

Naturally, our strategy is also contingent upon finding a suitable delay $l$.
There exist various techniques for determining the optimal delay size in the deterministic setting. 
However, these methods do not readily extend to IFSs.
We want to find the smallest $l>0$ such that each of the different delay maps are an embedding. 
However, popular methods for determining the minimum embedding dimension such as False Nearest Neighbours \cite{kennel1992determining}, or adaptations thereof \cite{cao1997practical}, are designed to check if a single delay map is an embedding, and so, can only be used when the base dynamics are known a priori. 
Therefore, we need a new methodology to check if multiple embeddings are present or not.
For the numerical examples presented in \cref{section: numerics}, we use an adhoc, brute force approach.
We search for the smallest delay which transforms the partial observations to a multi-manifold dataset by gradually increasing $l$. 
This works well for simple systems but for high dimensional and strongly nonlinear systems it may be nontrivial and remains an open problem beyond the aims of this paper.


\section{Time delay embedded Markov chains}\label{section: TDEMCS}
Following on from the previous section, our goal in this section is to discover the base dynamics $\omega$ that generated the observed time series.
In \cref{subsection:TDEMC-unravel}, we describe an algorithm that assigns to each delay map its corresponding $(l-1)$-tuple and in \cref{subsection: Algorithm proof} prove that this algorithm gives the correct output.
We use an abstract approach involving \textit{time delay embedded Markov chains} (TDEMCs). 

\begin{definition}[TDEMC]
    Let $(X_0, X_1,\dots)$ be a time homogeneous, irreducible Markov chain on $k$ symbols with transition matrix $P$.
    We define a time delay embedded Markov chain of delay $m>1$ as 
    \begin{equation}\label{eq:TDEMC} Y_n = X_nX_{n+1}\cdots X_{n+m-1}.\end{equation}
\end{definition}

Suppose we label the different delay maps $\Phi_{\psi,f,\underline{i}}$ in \cref{eq: delay embedding map} as $\textup{I},\textup{II},\dots,|\Sigma^{m}_P|$, where for notational convenience, we set $m:=l-1$.
Then each delay vector $\mathbf{z}_n(l)$ is assigned a label $Z_n\in\{\textup{I},\textup{II},\dots,\Sigma_P^{m}|\}$.
We let $\phi(i)$ denote the associated $m$-tuple of the $i^{\text{th}}$ delay map for $ i =\textup{I},\textup{II},\dots,|\Sigma^m_P|,$ and so $\phi(Z_n)=(\omega_n,\dots,\omega_{n+l-2})$. 
By learning $\phi(Z_n)$ we recover the original sequence of maps $\omega$ which generated the observation data, which we then to use to estimate the base dynamics' transition matrix $P$.

Using the sequence $Z_n$, we can estimate transition probabilities between different labels and store them in a transition graph.
The transition graph forms the basis of our method to recover $\omega$, outlined in \cref{fig: MC schematic} using a simple example.
From the topological structure of $G_Z$, we infer the topological structure of $G_X$.
Finally, we need to verify that the edge weights of $G_Z$ are compatible and consistent with possible edge weights on $G_X$, see for instance \cref{fig: MC schematic}.

\begin{figure}
    \centering
    \includegraphics[width=\textwidth, trim = 1cm 22.5cm 1cm 0cm]{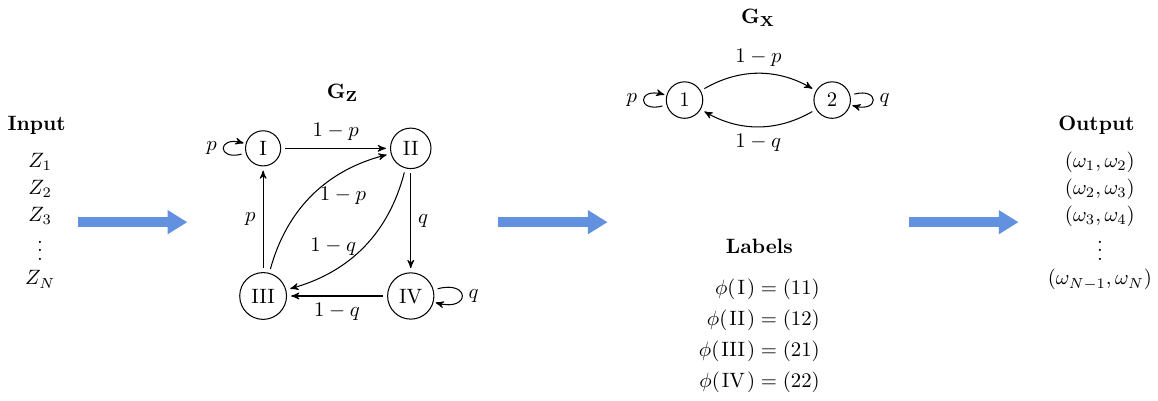}
    \caption{Summary of our procedure for recovering a time delay embedded Markov chain. 
    Using the input sequence $Z_n$ of delay maps, we estimate the transition graph $G_Z$ of the underlying TDEMC.
    After inferring $G_Z$, we apply \cref{alg: unembed MC} to discover the transition graph $G_X$ of the original MC that drives the ground truth IFS, along with the associated $(l-1)$-tuples of each delay map.
    Using $\phi$ we obtain the sequence of maps $\omega$ which generated the observation data.}
    \label{fig: MC schematic}
\end{figure}

\subsection{Procedure to unravel a TDEMC}\label{subsection:TDEMC-unravel}
We use the following set up. Let $(X_0, X_1,\dots)$ be a time homogeneous, irreducible Markov chain on $k$ symbols with transition matrix $P$, and $Y_n$ be the corresponding TDEMC of delay size $m$.
Additionally, let $\phi: \{1,\dots,|\Sigma_P^m|\} \to \Sigma_P^m$ be a bijection.
We assume we are given information on the Markov chain $Z_n$ satisfying $\phi(Z_n) = Y_n$, for some TDEMC $Y_n$ whose delay size $m$ is known, but $\phi$, $X_n$ and $Y_n$ themselves are not known.
Our goal is to recover $X_n$ from the observed MC $Z_n$.

Let $G_X$ refer to the transition graph of $X_n$, and $G_Z$ refer to the transition graph of $Z_n$. 
Our approach for recovering the original Markov chain $X_n$ requires knowledge on the \textit{elementary circuits} of $G_Z$, which we define below.

\begin{definition}[Closed walk]
    A walk on a graph $G$ refers to a sequence of nodes $\{v_1,\dots,v_n\}$, $n\ge2$, such that for all $i=1,\dots,n-1$, there exists a directed edge in $G$ from $v_i$ to $v_{i+1}$.
    A walk is closed if $v_n = v_1$.
\end{definition}

\begin{definition}[Elementary circuit]
    A circuit is a closed walk, in which each \underline{edge} is traversed only once. 
    An elementary circuit is a closed walk, in which each \underline{node} is visited only once (except the starting node). 
\end{definition}
The set of all elementary circuits $\Xi$ of a graph can be obtained using Johnson's circuit finding algorithm \cite{johnson1975finding}.
Let $\Xi_X$ and $\Xi_Z$ refer to the sets of elementary circuits of $G_X$ and $G_Z$ respectively.
Going forwards, we shall omit writing the final vertex in a closed walk, since the final node in the walk is immediately determined by its initial node.

We use the standard notion of distance on directed graphs \cite{chartrand1997distance}.
\begin{definition}[Directed distance]
    Let $\mathcal{N}$ denote the set of nodes of a graph $G$ and $U,V\subset \mathcal{N}$. 
We define the directed distance from  $U$ to $V$ in a graph $G$, as
\begin{equation*}
d(U,V) = \underset{u\in U, v\in V}{\text{min}}d(u,v),
\end{equation*}
where $d(u,v)$ is the length of any shortest path starting at $u$ and ending at $v$. Note that in a directed graph $d(U,V)\not\equiv d(V,U)$.
\end{definition}

We also need the following properties on nodes $v\in G_Z$, which follow directly from the definition of $G_Z$.

\begin{proposition}
\label{obs: inferring labels}\hfill
    \begin{enumerate}
        \item \label{obs: period < m}
        Let $\xi\in\Xi_Z $ be an elementary circuit of length $l < m$. Let $\{v_{1},\dots,v_{l}\}$ denote the nodes of $\xi$.
        Let $\phi(v_i) = a_{i_1}\cdots a_{i_m}$.
        Then for all $v_i \in \xi$, for all $1 \le j \le m -l $,  $$a_{i_j} =  a_{i_{j + l}}.$$
        \item \label{obs: dist < m}
        Fix $u,v \in G_z$ and let $\phi(u) = a_1\cdots a_m$.
        \begin{enumerate}
            \item If $d(v,u) < m$ then $\phi(v) = (\cdots, a_1,\dots,a_{m-d(v,u)})$.
            \item If $d(u,v) <m$ then $\phi(v) = (a_{d(u,v)+1},\dots,a_m,\cdots)$.
        \end{enumerate}
    \end{enumerate}
\end{proposition}

The procedure for recovering $X_n$ from $Z_n$ is as follows. 
We use \cref{alg: unembed MC} to determine $\phi$ and hence recover $Y_n$.  
From $Y_n$ we obtain the original Markov chain $X_n$ (up to a possible relabelling). 

\begin{algorithm}
    \caption{Unembedding a TDE Markov chain}\label{alg: unembed MC}
    \begin{algorithmic}
        \STATE \textbf{Inputs.}
        \STATE \hskip1em $1.$ Delay size $m\in\mathbb{N}$.
        \STATE \hskip1em $2.$ List of elementary circuits $\Xi = \left[\xi_1,\dots,\xi_S\right]$ in $G_z$ sorted by length, from shortest to
        \STATE \hskip2.1em longest.
        \STATE \hskip1em $3.$ Weighted adjacency matrix $Q$ of $G_z$.
        \STATE \textbf{Ouputs.}
        \STATE \hskip1em $1.$ A unique $m$-tuple for each node $v\in G_Z$.
        \STATE \hskip1em $2.$ The unembedded graph $G_X$.
        \STATE \hskip1em $3.$ Weighted adjacency matrix $P$ of $G_X$.
        \STATE \textbf{Initialise.}
        \STATE \hskip1em $1.$ A null graph $G_X$, i.e., a graph with 0 nodes and 0 edges.
        \STATE \hskip1em $2.$ $L = [\ ]$. A list storing a subset of nodes in $G_Z$ which belong to elementary cycles and 
        \STATE \hskip2.1em have an assigned tuple.

        \STATE \textbf{Steps.}\quad For $k=1,\dots,S$:
        \STATE \hskip1em $1$. Remove $\xi_k$ from $\Xi$. Let $\{v_{1},\dots,v_{n_k}\}$ denote the nodes of $\xi_k$ in order along the 
        \STATE \hskip2.2em circuit.
        \STATE \hskip1em $2.$ Assign a tuple to each state in $\xi_k$ using \cref{Alg: MC labels}.
        \STATE \hskip1em $3.$ Let $\beta(v_i)$ denote the last element of its associated tuple (of length $m$).
        \STATE \hskip2.1em $3.1$ If the closed walk $\{\beta(v_1),\dots,\beta(v_{n_k})\} \subseteq G_X$, do not update $L$.
        \STATE \hskip2.1em $3.2$ Else, add the elementary circuit $\{\beta(v_1),\dots,\beta(v_{n_k})\}$ to $G_X$. 
        \STATE \hskip3.7em For each $i=1,\dots,n_k$, set $P_{\beta(v_i),\beta(v_j)} = Q_{v_i,v_j}>0,$ where $j := (i+1)\text{mod} \ n_k$.
        \STATE \hskip3.7em Append to $L$ the set  $\{v_{i} : v_{i}\notin L \}$. 
    \end{algorithmic}
\end{algorithm}

\begin{theorem}\label{thm: 2}
  \cref{alg: unembed MC} outputs $G_X$, $\phi$.
\end{theorem}
We defer the proof of \cref{thm: 2} to \cref{subsection: Algorithm proof}.

\begin{remark}
    It suffices to consider only a minimum circuit basis (see \cite{gleiss2003circuit} for details) of $G_Z$ as input to \cref{alg: unembed MC} rather than all elementary circuits of $G_Z$.
This may reduce the size of the computation.
\end{remark}

\begin{algorithm}
    \caption{Assigning $m$-tuples to an elementary circuit}\label{Alg: MC labels}
    \begin{algorithmic}
        \STATE \textbf{Inputs.}
        \STATE \hskip1em $1.$ Delay size $m\in\mathbb{N}$.
        \STATE \hskip1em $2.$ Elementary circuit $\xi_k = \{v_{k_1},\dots,v_{k_{n_k}}\}\subset G_Z$, with the nodes of $\xi_k$ listed in order
        \STATE \hskip2.1em along the circuit.
        \STATE \hskip1em $3.$ A list $L$ containing a subset of nodes in $G_Z$ which belong to elementary cycles and
        \STATE \hskip2.1em have an assigned tuple.
        \STATE \hskip1em $4.$ Local variable $N_V$ which counts the number of nodes in $G_X$.
        \STATE \textbf{Ouputs.}
        \STATE \hskip1em $1.$ A unique $m$-tuple for each of the nodes $v_{k_i}\in \xi_k$.
        \STATE \textbf{Steps.}
        \STATE \hskip1em $1.$ If $L\neq \emptyset$, for each $v_{k_i}\in\xi_k$ with $d(v_{k_i},L)<m$ or $d(L, v_{k_i})<m$, update its tuple
        \STATE \hskip2.1em according to \hyperref[obs: dist < m]{\cref*{obs: inferring labels}.\ref*{obs: dist < m}}. 
        \STATE \hskip1em $2$ While there are nodes $v_{k_i}$ with incomplete tuples:
        \STATE \hskip2.1em $2.1$ Take any node $v_{k_j}$with an incomplete tuple. 
        \STATE \hskip3.7em Set the first missing element in its tuple to $N_V+1$.
        \STATE \hskip2.1em $2.2$ Update the tuple of $v_{k_j}$ according to \hyperref[obs: dist < m]{\cref*{obs: inferring labels}.\ref*{obs: period < m}}.
        \STATE \hskip2.1em $2.3$ Update the tuples of the other nodes in $\xi_k$ using \hyperref[obs: dist < m]{\cref*{obs: inferring labels}.\ref*{obs: dist < m}}.
        \STATE \hskip2.1em $2.4$ Set $N_V = N_V +1$.
    
    \end{algorithmic}
\end{algorithm}

We demonstrate \cref{alg: unembed MC} on the TDE Markov chain whose transition graph $G_Z$ is shown in \cref{fig: MC schematic} and has delay size $m=2$.
The elementary circuits of $G_Z$, sorted by length, are 
\begin{equation*}\Xi = \left\{\{\textup{1}\},\{\textup{IV}\}, \{\textup{II,III}\}, \{\textup{I,II,III}\}, \{\textup{II,IV,III}\}, \{\textup{I,II,IV,III}\}\right\}.
\end{equation*}
Applying \cref{alg: unembed MC}, we take the first circuit in $\Xi$. Since this circuit has length 1, we assign to state $1$ the tuple $(11)$ and append the state to $L$.
Next, we take the circuit $\{\textup{IV}\}$. As $d(\{\textup{IV}\},\{\textup{I}\})=d(\{\textup{I}\},\{\textup{IV}\})\not< m$ and $|\{\textup{IV}\}|=1$, we assign to state $\textup{IV}$ the tuple $(22)$ and append this state also to $L$.
The next smallest circuit in $\Xi$ is $\{\textup{II,III}\}$. 
Using existing knowledge on assigned tuples to states in $L$ and \hyperref[obs: dist < m]{\cref*{obs: inferring labels}.\ref*{obs: dist < m}}, we infer that state $\textup{II}$ is associated to the tuple $(12)$ and state $\textup{III}$ is associated to $(21)$.
Having assigned tuples to each of the states in $G_Z$, by taking into account the edge weights of the embedded graph, we obtain the transition graph $G_X$ shown in \cref{fig: MC schematic}.

\subsection{Proof of \cref{thm: 2}}\label{subsection: Algorithm proof}
We begin by showing that the TDE Markov chain is irreducible if and only if the original Markov chain is irreducible.

\begin{proposition}\label{prop: irreducibility}
    $X_n$ is irreducible iff $Z_n$ is irreducible.
\end{proposition}

\begin{proof}
    ($\implies$)
    Suppose $X_n$ is irreducible.
    Fix nodes $u,v\in G_Z$. 
    WLOG write $\phi(u) = a_1\cdots a_m$ and $\phi(v) = b_1\cdots b_m$. 
    We first show $v$ is accessible from $u$. 
    Since $X_n$ is irreducible, we know $b_1$ is accessible from $a_m$ in $X$. 
    Hence, there exists $n\in\mathbb{N}$ such that
    \begin{equation*}
    \mathbb{P}(X_{n} = b_1|X_m = a_m, \dots, X_1 = a_1)>0.
    \end{equation*}
    Since $v\in G_Z$, we have $\mathbb{P}(X_m = b_m, X_{m-1} = b_{m-1},\dots X_2 = b_2|X_1=b_1)>0$.
    Hence, by homogeneity and the Markov property, $\exists n \in \mathbb{N}$ such that 
    $\mathbb{P}(Z_n = b|Z_1= a) >0$.
    Analogously, we can show  that $a$ is accessible from $b$ and hence $Z_n$ is irreducible.

    ($\impliedby$)
    Suppose $X_n$ is not irreducible.
    By definition, there exist nodes $a_m,b_m \in G_X$ such that 
    $$\mathbb{P}(X_n = b_m|X_1=a_m) = 0 \quad \forall n >0.$$
    Fix nodes $a_1,\dots, a_{m-1}, b_1,\dots, b_{m-1} \in G_X$ such that $\mathbb{P}(X_2 = a_2,\dots, X_m=a_m|X_1 = a_1)>0$
    and $\mathbb{P}(X_2= b_2,\dots,X_m = b_m|X_1=b_1)>0$. 
    Then there exists nodes $\phi^{-1}(a_1a_2\cdots a_m),$ $\phi^{-1}(b_1b_2\cdots b_m) \in G_Z$, however, 
    $$\mathbb{P}(Z_n = \phi^{-1}(b_1b_2\cdots b_m) |Z_1 = \phi^{-1}(a_1a_2\cdots a_m) ) = 0 \quad \forall n >1.$$
    Therefore, $Z_n$ is not irreducible.
    This completes the proof.
\end{proof}

As $X_n$ is irreducible by assumption, then by \cref{prop: irreducibility}, $Z_n$ is also irreducible. 
By definition of irreducibility, this implies that the graphs $G_X$ and $G_Z$ are both strongly connected. 
As $G_X$ is strongly connected, each edge in $G_X$ belongs to an elementary circuit \cite{gleiss2003circuit}. 
Therefore, to recover $G_X$, it is sufficient to identify each of the elementary circuits in $G_X$.

The next two propositions describe the relationship between elementary circuits in the two topological graphs $G_X$ and $G_Z$.

\begin{proposition}
    The number of elementary circuits in $G_Z$ is greater than or equal to the number of elementary circuits in $G_X$.
\end{proposition}

\begin{proof}
    Let $\{a_1,a_2,\dots,a_p\}$ be an elementary circuit of size $p$ in $G_X$. 
    Note that necessarily $p\le k$, the number of states in the original Markov chain.
    Let $v_i$ satisfy 
    \begin{equation*}
    \phi(v_i) := a_i a_{(i+1) \hskip0.1em \text{mod} \hskip0.1em p }\hskip0.1em a_{(i+2) \hskip0.1em \text{mod} \hskip0.1em p }\cdots a_{(i+m-1) \hskip0.1em \text{mod} \hskip0.1em p },
    \end{equation*}
    for each $i=1,\dots,p$. 
    Then, by construction, $\{v_1,\dots,v_p\}$ is a elementary circuit of $G_Z$.
\end{proof}

\begin{proposition}\label{prop: induced elementary circuits}
    Elementary circuits in $G_Z$ are formed by delay embedding closed walks in $G_X$ which are not necessarily elementary circuits. 
\end{proposition}
\begin{proof}
    Let $\xi = \{v_1,\dots,v_s\}$ be an elementary circuit in $G_Z$ and 
    let $\phi(v_i):= a_{i_1}\cdots a_{i_m}$ for each $i=1,\dots,s$. 
    Then it follows that $\{a_{1_1},\dots, a_{s_1}\}$ is a closed walk on $G_X$.
    If the closed walk is not an elementary circuit in $G_X$ it must be a combination of two or more elementary circuits in $G_X$ \cite[Lemma 2]{gleiss2003circuit}.
\end{proof}

As an example, consider again the TDE Markov chain depicted in \cref{fig: MC schematic}.
$G_X$ admits three elementary circuits: $\{1\},$ $\{2\},$ $\{1,2\}$, while $G_Z$ admits six elementary circuits: $$\{\textup{1}\},\{\textup{IV}\}, \{\textup{II,III}\}, \{\textup{I,II,III}\}, \{\textup{II,IV,III}\}, \{\textup{I,II,IV,III}\},$$ where
$$ \phi(11) = \textup{I},\ \phi(12)=\textup{II},\ \phi(21)=\textup{III}, \ \phi(22) = \textup{IV}.$$
The extra elementary circuits of $G_Z$ correspond to the (delay embedded) closed walks $\{1,1,2\},$ $\{1,2,2\},$ $\{1,1,2,2\}$ on $G_X$.

Following \cref{prop: induced elementary circuits}, we shall say that an elementary circuit $\xi_Z$ in $G_Z$ corresponds to a closed walk $\xi_X$ in $G_X$ if embedding the path $\{\xi_X,\xi_X,\dots\}$ on $G_X$ forms the path $\{\xi_Z,\xi_Z,\dots\}$ on $G_Z$, and vice versa.  
It remains to find the elementary circuits in $G_Z$ that also correspond to elementary circuits in $G_X$.
Some of the shared elementary circuits of the two graphs can be easily identified, these circuits are characterised in \cref{obs: min length}.

\begin{observation}\label{obs: min length}
    Let $l$ denote the length of the smallest circuit(s) in $G_X$. If a circuit $\xi \in G_Z$ has length $l$, then $\xi\in\Xi_Z$.
\end{observation}
\begin{proof}
    This follows directly from \cref{prop: induced elementary circuits}. 
    If $\xi$ does not correspond to an elementary circuit in $G_X$ it must be a combination of two or more elementary circuit in $G_X$ and therefore have length greater than $l$.     
\end{proof}

We can naturally extend \cref{obs: min length} as follows.
Let $\Pi$ denote the set of elementary circuits in $G_Z$ which also correspond to elementary circuits in $G_X$.
Denote by $\Pi_l$ the set of elementary circuits in $\Pi$ of length $\le l$.

\begin{corollary}\label{corollary: composition of circuits}
    Let $\xi^* \in \Xi_Z$ with $|\xi^*| = \text{min}_{|\xi|}\{ \xi \in \Xi_Z : |\xi| > l\}$. 
    Then either $\xi^* \in \Pi$ or $\xi^*$ is a combination of two or more elementary circuits in $\Pi_l$.
\end{corollary}

Lastly, we need the following proposition, which states that if two elementary circuits in $G_Z$ are separated by a distance of at least $m$, then the two corresponding closed walks in $G_X$ do not share any common nodes.
\begin{proposition}\label{prop: dist >= m}
Let $\xi_X$ and $\xi'_X$ be closed walks on $G_X$. Suppose $\xi_X$ and $\xi'_X$ share no states in common. 
Let $\xi_Z,\xi'_Z$ be the corresponding walks embedded on $G_Z$ of $\xi_X$ and $\xi'_X$ respectively. 
Then $d(\xi_Z,\xi'_Z)\ge m.$  
\end{proposition}

\begin{proof}
    Suppose for contradiction, that there exists $u\in\xi_Z, v\in\xi'_Z$ such that $d(u,v) <m$. 
    Without loss of generality, let $\phi(u) = a_1\cdots a_m$ and $\phi(v) = b_1\cdots b_m$.
    From \hyperref[obs: dist < m]{\cref*{obs: inferring labels}.\ref*{obs: dist < m}}, we have $b_1 = a_{d(u,v)+1}$
    and, hence, $\xi_X$ and $\xi'_X$ must share at least one state in common. 
    This contradicts the initial assumption, we therefore conclude that $d(u,v) \ge m$.
\end{proof}

We combine the above results to prove \cref{thm: 2}.
\begin{proof}[Proof of \cref{thm: 2}]
    Every edge in $G_X$ belongs to at least one elementary circuit.
    In the proof of \cref{prop: induced elementary circuits} we show that every elementary circuit in $G_X$ also appears embedded in $G_Z$.
    It follows by induction using  \cref{obs: inferring labels}, \cref{corollary: composition of circuits} and \cref{prop: dist >= m} that \cref{alg: unembed MC} is able to uniquely identify every edge in $G_X$ and recover $\phi$, after a possible relabelling, from elementary circuits of $G_Z$.
\end{proof}


\section{Identifying diffeomorphic iterated function systems}\label{section: identifying conjugate mappings}
We discuss the final step in our methodology: how to construct generators for an IFS that is bundle diffeomorphic to the ground truth IFS.
As established in \cref{section: multi-manifold clustering,section: TDEMCS}, we can learn the different generators in \cref{eq: delay dynamics} and their underlying symbol structure.  
We want to learn a minimal number of generators, equal to the number of generators in the ground truth model. 
It is tempting to try learning such a model by reformulating each mapping $F_{\omega_n,\dots,\omega_{n+l-1}}$ in \cref{eq: delay dynamics} as a composition
$g_{\omega_{n+l-1}}\circ\cdots\circ g_{\omega_n}: \mathbb{R}^l \to \mathbb{R}^l$.
In general this is not possible because the non-autonomous dynamics cannot be disentangled in the delay space.

However, we can learn generators $g_i$ that are compatible with the partial observations.
Following \cref{section: TDEMCS}, we start with the knowledge of the sample $\omega$ associated to the time series of partial observations.
In \cref{subsection: dynamical equivalence} we investigate the relationship between the ground truth IFS and other IFSs that satisfy \cref{eq: delay dynamics} and in \cref{subsection: HDI} we discuss a particular method for finding such models.

\subsection{Dynamical Equivalence} \label{subsection: dynamical equivalence}
Suppose there exists an IFS, different to the ground truth IFS but with the same Markov chain in the base, defined on some compact manifold $M'$ and a function $\psi':M' \to\mathbb{R}$ such that the new observable and ground truth IFS observations both evolve according to \cref{eq: delay dynamics}.
In other words, we assume that both IFSs satisfy the same ``reduced model'' \cite{stepaniants2024discovering}. 
In this case, for generic observations the two systems are bundle diffeomorphic. 
This follows directly from Takens' theorem for IFSs \cite{stark}.

\begin{theorem}\label{thm: HDI}
Consider two IFSs 
\begin{equation*}x_{n+1} = f_{\omega_n}(x_n), \quad y_{n+1} = g_{\omega_n}(y_n),\ n = 0,1,2,\dots
\end{equation*}
where $f_i:M\to M,  g_i:M'\to M'$ for each $i=1,\dots,k$ and $\ \omega_0\omega_1\omega_2\dots\in\Sigma^+_P$. 
 Let \newline $\Omega_n = (\omega_n,\dots,\omega_{n+l-1})$.
Suppose there exists observation functions $\psi: M\to \mathbb{R},$ \newline
$\psi': M'\ \to\mathbb{R}$ and $l\ge 1$ such that $\Phi_{\psi,f,\Omega_{n}}$ and $\Phi_{\psi',g,\Omega_n}$ are embeddings of $M$ and $M'$ resp. and 
that
\begin{equation*}\Phi_{\psi,f,\Omega_{n+1}} \circ f_{\omega_n}\circ\Phi_{\psi,f,\Omega_n}^{-1} = \Phi_{\psi',g,\Omega_{n+1}} \circ g_{\omega_n}\circ\Phi_{\psi',g,\Omega_n}^{-1}.\end{equation*}
 for each admissible $\Omega_n$.
 Then the both systems reduce to \cref{eq: delay dynamics} and hence they are bundle diffeomorphic.
\end{theorem}

\begin{proof}
Let $F$ denote the delay reconstructed system \cref{eq: delay dynamics}. 
Using Takens theorem for IFSs, we have 
$(f,\omega) \longleftrightarrow (F, \omega),$
with $\longleftrightarrow$ denoting the equivalence relation of (smooth) bundle conjugacy \cite{stark}.

Similarly, we have 
$(g,\omega) \longleftrightarrow (F, \omega).$
Applying transitivity of $\longleftrightarrow$, we obtain
\begin{equation*}(f,\omega)\longleftrightarrow (g,\omega).
\end{equation*}
We thus have the following conjugacy diagram:

\begin{center}
    \begin{tikzpicture}[>=triangle 60]
        \matrix[matrix of math nodes,column sep={120pt,between origins},row
          sep={60pt,between origins},nodes={asymmetrical rectangle}] (s)
        { |[name=X1]| M &|[name=X2]| M\\
          |[name=Z1]| \mathbb{R}^l &|[name=Z2]| \mathbb{R}^l\\
          |[name=Y1]| M' &|[name=Y2]| M'\\
        };
        \draw[->] (X1) edge node[auto] {\(f_{\omega_0}\)} (X2)
                (Y1) edge node[auto] {\(g_{\omega_0}\)} (Y2)
                (Z1) edge node[auto] {\(F_{\omega_0,\dots,\omega_{l-1}}\)} (Z2)
                (Y1) edge node[auto] {\(\Phi_{g,\psi',\Omega_0}\)} (Z1)
                (Y2) edge node[right] {\(\Phi_{g,\psi',\Omega_1}\)} (Z2)
                (X1) edge node[left] {\(\Phi_{f,\psi,\Omega_0}\)} (Z1)
                (X2) edge node[auto] {\(\Phi_{f,\psi,\Omega_1}\)} (Z2)
      ;
    \end{tikzpicture}
    \end{center}
\end{proof}

\begin{remark}
    \cref{thm: HDI} and its proof hold analagously for autonomous deterministic dynamical systems with (smooth) topological conjugacies replacing (smooth) bundle conjugacies.
\end{remark}

We finally show that the results in \cref{section: TDEMCS,section: identifying conjugate mappings} above constitute a proof of our main result in \cref{th:main} for an IFS satisfying \cref{hyp: embedding} and the idealisation of unlimited data.

\begin{proof}[Proof of \cref{th:main}]\label{proof: main}
    the first thing we observe is that in a suitable delay space, we can separate the time series into different delay maps. 
    Through section 3 we can identify the original driving signal.
    From Section 4 we can associate the driving signal to a set of generators that is bundle diffeomorphic to the ground truth IFS.

    The final assumption in \cref{hyp: embedding} ensures that the different delay maps are topologically distinct.
    Therefore in theory, given unlimited data, the different delay maps can be separated (with any ambiguous delay vectors which lie in the intersection of two manifolds being removed).
    It follows from \cref{thm: 2} that after separating the different delay maps, we can recover the base dynamics that generated the data.
    Using optimisation techniques, such as those described in \cref{subsection: HDI}, we can discover a set of generators whose evolution in the delay coordinates reduces to \cref{eq: delay dynamics}.
    \Cref{thm: HDI} implies that learnt system is bundle diffeomorphic to the ground truth IFS.

\end{proof}

\subsection{Hidden Dynamics Inference}\label{subsection: HDI}
Inspired by \cite{stepaniants2024discovering}, we use hidden dynamic inference to identify IFS decompositions of the delay embedded system that satisfy \cref{eq: delay dynamics}.

We define a discrete time, multi-functional HDI model with $V$ latent variables and $k$ different functions as follows:
\begin{equation}\label{eg: HDI}\begin{aligned}
    z_{n+1} &= g_{\omega_n}(\omega_n; z_n,h^1_n,\dots,h^V_n), \\
    h^i_{n+1} &= g^i_{\omega_n}(\omega_n; z_n,h^1_n,\dots,h^V_n), \quad i = 1,\dots,V,    
\end{aligned}
\end{equation}
where $\omega_n \in \{1,\dots,k\}$. 

As proposed in \cite{stepaniants2024discovering}, the initial condition can be included in the set of parameters to be optimised. Otherwise, the first few  entries of the observed time series can be used instead.
Additionally, each function in the model is assumed to be a linear combination of some pre-defined basis functions whose coefficients minimise the Mean Square Error (MSE) loss defined below, and are found using gradient-based optimization and automatic differentiation.

Let $\mathbf{p}$ represent the discrete time multifunctional HDI model parameters.
We define the MSE loss, given an input sequence of labels $\underline{\omega} = \{\omega_n\}_{n=0}^{N-1}$ as
\begin{equation}\label{Eq: HDI MSE}
    \text{MSE}(\mathbf{p}) = \sum_{n=1}^{N-1} || g_{\omega_{n}}(z_n,h^1_n,\dots,h^V_n) - z_{n+1}||^2.
\end{equation}
By minimising the MSE and by \cref{thm: HDI}, we effectively learnt an IFS model that is diffeomorphic to the ground truth IFS.

For a partially observed IFS, we set $V+1$ equal to the (maximum) estimated dimension of the sub-manifolds $\Phi_{\psi,f,\underline{i}}(M)$ in \cref{eq: delay embedding map}.
However, models that preserve the dynamics of the observed variable can still be found for various values of $V$.


\section{Numerical Results}\label{section: numerics}
In this section we illustrate our methodology on some numerical examples.
We demonstrate the ability of the proposed framework to discover bundle diffeomorphic systems from time series of partial observations from two different IFSs.
Both systems have nonlinear dynamics. 
The first example uses a Mobius IFS, the second features a hyperbolic IFS.
The results can be reproduced using the source code in \cite{sourcecode}.

\subsection{Curvilinear Sierpinski gasket}
We first study the following Mobius IFS with compact domain $M = \{ z \in \mathbb{C} : |z| \le 1\}$. 
The IFS consists of the following generators, 
\begin{equation*}
  \{f, R\circ f, R^2\circ f\},
\end{equation*}
where $f(z) = \frac{(\sqrt{3}-1)z+1}{-z+\sqrt{3}+1}$ and $R(z)=e^{\nicefrac{2\pi i}{3}}z$. 
The attractor of the IFS is referred to as the curvilinear Sierpinski gasket. 
We assume that each map in the IFS has an equal probability of being chosen and use observation function $\psi(z) = \text{Im}(z)$. 
The entire simulation is shown in \cref{fig: results Apollonian Gasket IFS B}.

We first embed the observations in $\mathbb{R}^3$, and so there are 9 different delay maps which we separate using multi-manifold clustering.
The delay embedding is shown in \cref{fig: Embedded Apollonian Gasket}.
We then unembed the TDE Markov chain and estimate the original Markov chain which generated the data using \cref{alg: unembed MC}. 
Lastly, we use HDI to find an analytic model representation for the observations. We are successfully able to identify an IFS comprised of 3 functions that is approximately bundle diffeomorphic to the original system. 
For this system we introduce a single hidden variable and use a basis of cubic polynomials.
A sample trajectory from the learnt IFS model is displayed in \cref{fig: results Apollonian Gasket IFS C}.

\begin{figure}
        \begin{subfigure}[b]{\linewidth}
            \centering
            \includegraphics[width=0.5\textwidth]{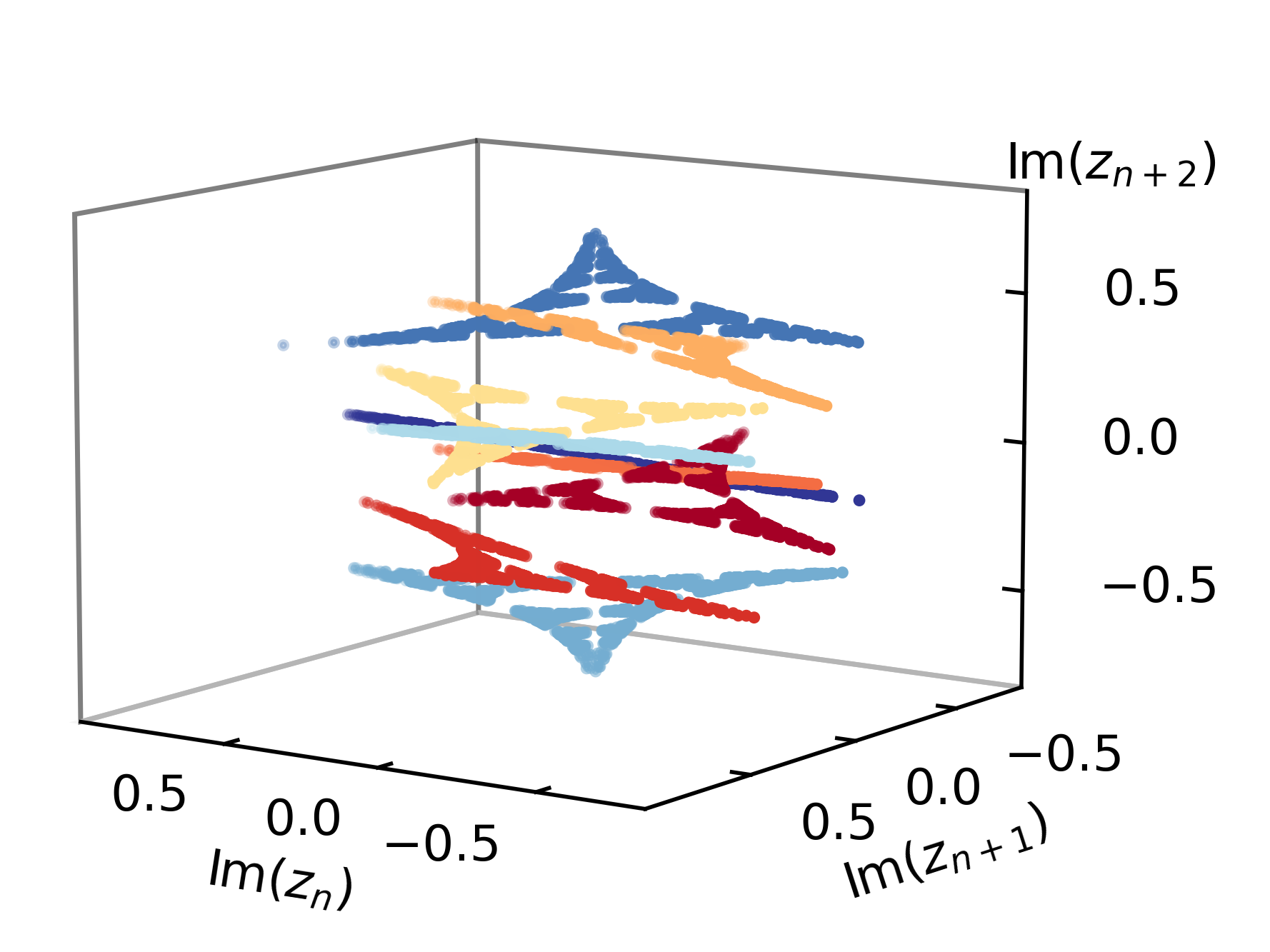}
            \caption{Delay embedding of partially observed Curvilinear Sierpinski IFS. Different colours are used to identify the different delay maps.}
            \label{fig: Embedded Apollonian Gasket}
        \end{subfigure}
        \begin{subfigure}[b]{0.5\linewidth}
            \centering
            \includegraphics[width=\textwidth]{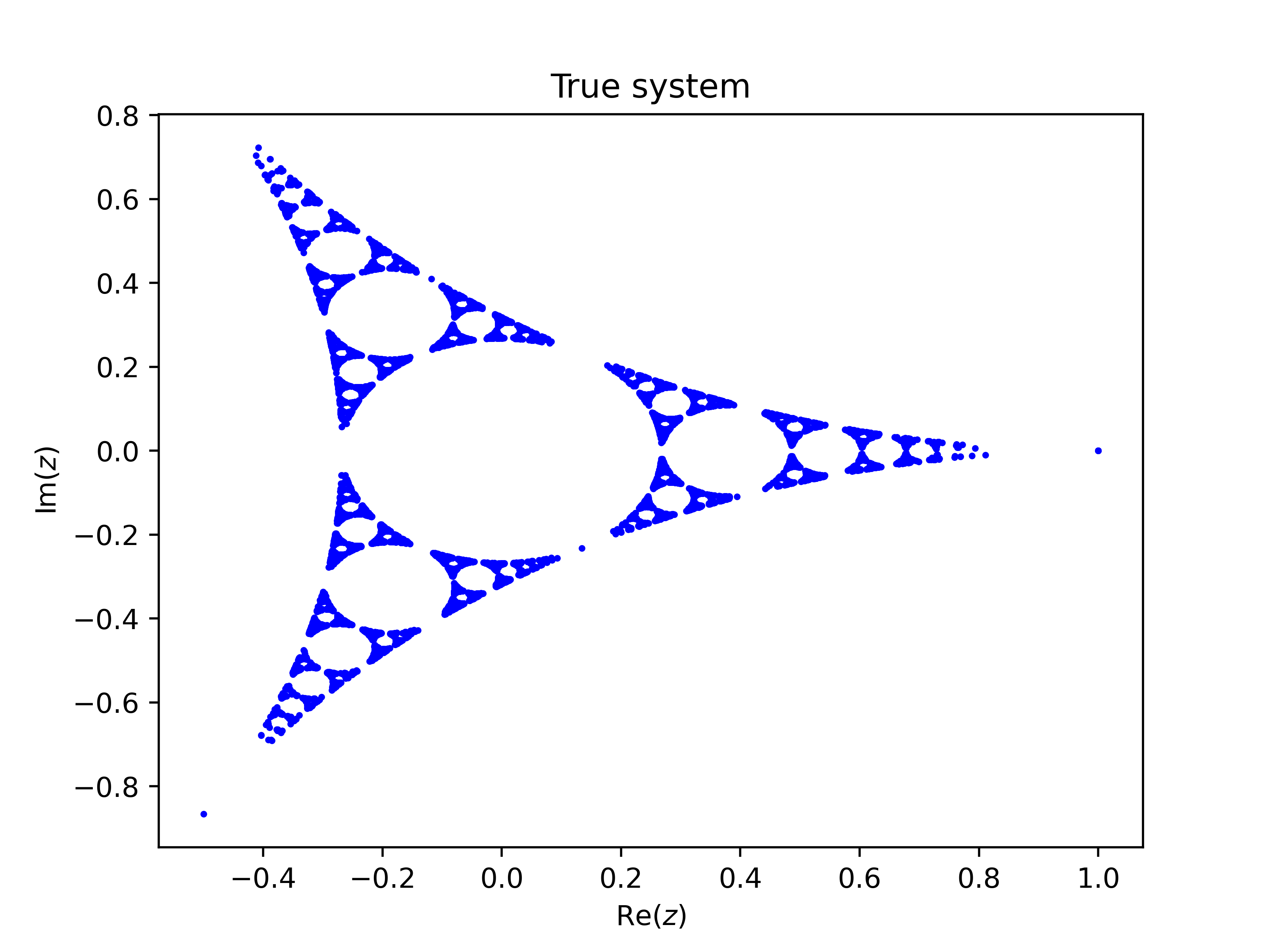}
                \caption{Original system}
                \label{fig: results Apollonian Gasket IFS B}
        \end{subfigure}
        \begin{subfigure}[b]{0.5\linewidth}
            \centering
            \includegraphics[width=\textwidth]{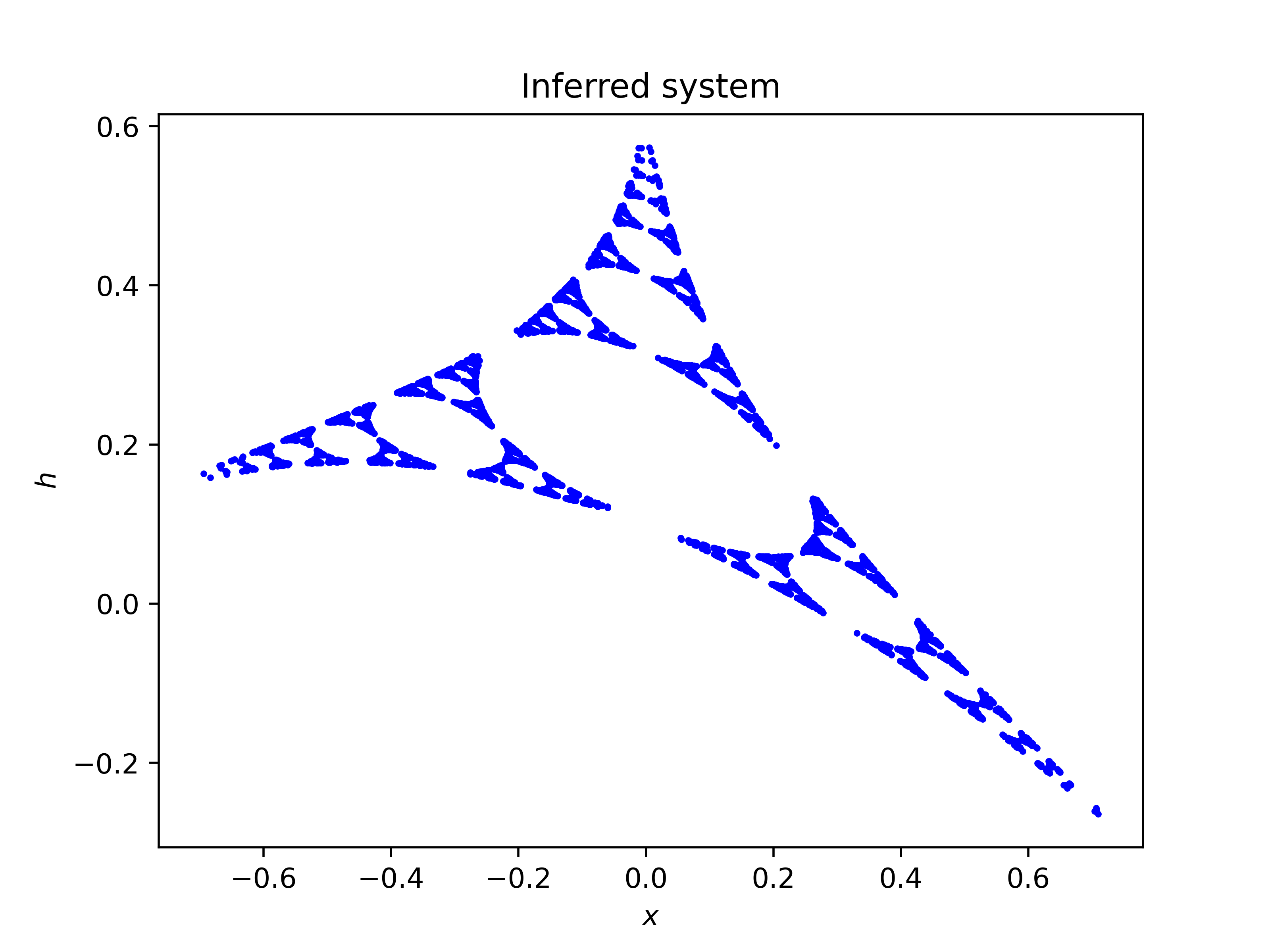}
                \caption{Learnt system, MSE$(\textbf{p}) = 1.3\times 10^{-5}$}
                \label{fig: results Apollonian Gasket IFS C}
        \end{subfigure}
        \caption{Results from the partially observed curvilinear Sierpinski gasket IFS.}
        \label{fig: results Apollonian Gasket}
        \end{figure}

\subsection{H\'enon IFS} \label{subsection: Henon IFS}
We consider the following H\'enon IFS, adapted from \cite{detectingIFS}. The generators of the IFS are:
\begin{equation}
             f_1(x,y) = (y + 1 - 1.2x^2, \ 0.3x), \quad
             f_2(x,y) = (y + 1 - 1.2(x-0.2)^2, \ -0.2x).
\end{equation}
We let $\psi(x,y) = x$ be the observation function and simulate a sample trajectory, shown in \cref{fig: results Henon IFS A}.
For visualisation purposes, the different H\'enon maps in the IFS are displayed in different colours.

Once again, we embed the observed time series in $\mathbb{R}^3$, where we are able to cluster and separate $4$ different $2$-dimensional sub-manifolds and obtain an estimate the original sequence of maps used.
Here, we take the basis functions to be quadratic polynomials and create one hidden variable.
The results of a sample trajectory from the learnt model are shown in \cref{fig: results Henon IFS B}.

\begin{figure}
    \begin{subfigure}[b]{0.5\linewidth}
        \centering
        \includegraphics[width=\textwidth]{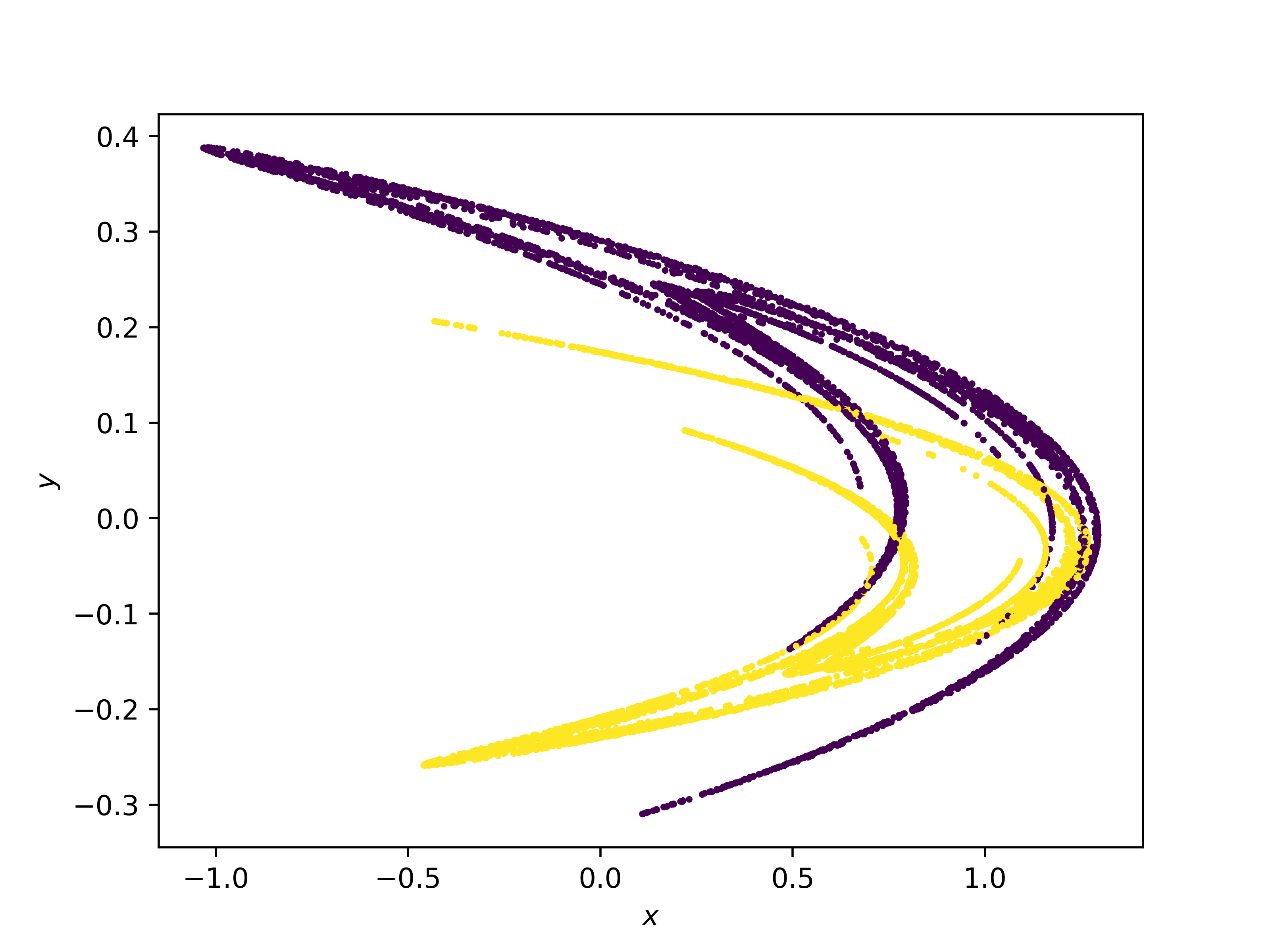}
            \caption{Original system}
            \label{fig: results Henon IFS A}
    \end{subfigure}
    \begin{subfigure}[b]{0.5\linewidth}
        \centering
        \includegraphics[width=\textwidth]{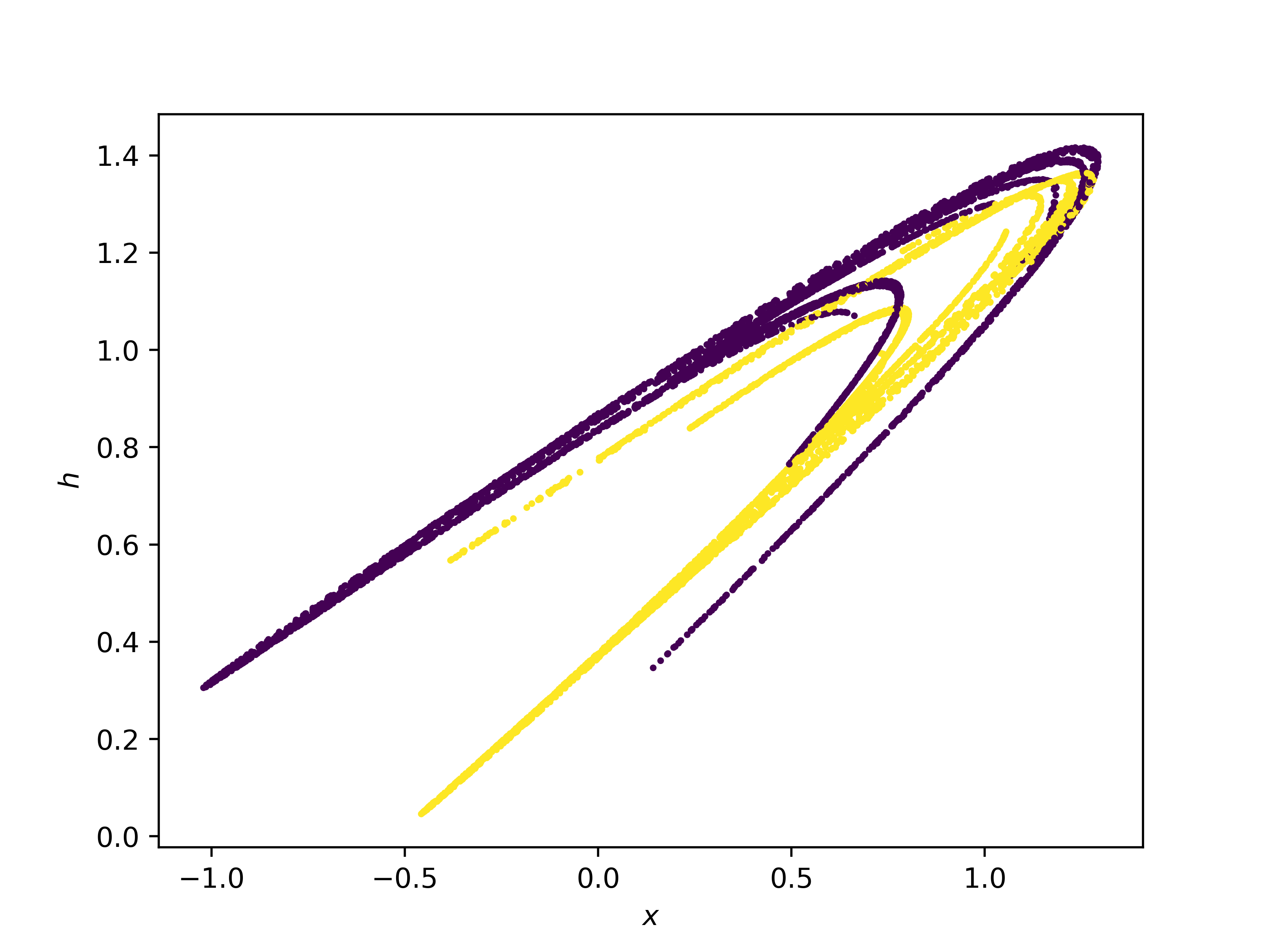}
            \caption{Learnt system, MSE$(\textbf{p}) = 4.5\times 10^{-5}$}
            \label{fig: results Henon IFS B}
    \end{subfigure}
    \caption{Sample trajectories from the ground truth H\'enon IFS and the learnt IFS model.}
    \label{fig: results Henon IFS}
    \end{figure}


\section{Conclusions}\label{section: conclusions}
In this work we develop a data-driven methodology for learning finitely generated iterated function systems from time series data of partial observations.
Although we are motivated by learning random systems, our methodology is directly applicable to non-autonomous systems with a finite number of generators.
Our approach draws upon results and methods from various different areas of mathematics including, time-delay embeddings, manifold learning, directed graphs and Markov chains and model identification.
We highlight differences between delay embeddings of autonomous deterministic systems and partial observations of IFSs.
This study is the first step towards generalising model discovery toolkits to broader classes of random dynamical systems.


\section*{Acknowledgments}
The authors are grateful to Matthew Levine, Matheus M. Castro and Kevin Webster for useful discussions and acknowledge support from the EPSRC Centre for Doctoral Training in Mathematics of Random Systems: Analysis, Modelling and Simulation (EP/S023925/1). 
The research of JSWL has furthermore been supported by EPSRC grants EP/W009455/1,  EP/Y020669/1, EP/Y028872/1 and EP/Z533658/1, as well as by the JST (Moonshot R \& D Grant Number JPMJMS2021) and GUST (Kuwait).


\bibliographystyle{plain}
\bibliography{references}

\begin{thebibliography}{10}

\bibitem{detectingIFS}
Zachary Alexander, James~D. Meiss, Elizabeth Bradley, and Joshua Garland.
\newblock Iterated function system models in data analysis: Detection and separation.
\newblock {\em Chaos}, 22(023103), 2012.

\bibitem{barnsley2012fractals}
Michael~F. Barnsley.
\newblock {\em Fractals everywhere}.
\newblock Academic Press, 2 edition, 1993.

\bibitem{IFSchannels}
D.~S. Broomhead, J.~P. Huke, M.~R. Muldoon, and J.~Stark.
\newblock Iterated function system models of digital channels.
\newblock {\em Proc. A.}, 460:3123--3142, 2004.

\bibitem{cao1997practical}
Liangyue Cao.
\newblock Practical method for determining the minimum embedding dimension of a scalar time series.
\newblock {\em Phys. D}, 110(1):43--50, 1997.

\bibitem{chartrand1997distance}
Gary Chartrand and Songlin Tian.
\newblock Distance in digraphs.
\newblock {\em Comput. Math. Appl.}, 34(11):15--23, 1997.

\bibitem{sourcecode}
Emilia Gibson.
\newblock Detecting-\uppercase{IFS}, 2025.

\bibitem{gleiss2003circuit}
Petra~M. Gleiss, Josef Leydold, and Peter~F. Stadler.
\newblock Circuit bases of strongly connected digraphs.
\newblock {\em Discuss. Math. Graph Theory}, 23:241--260, 2003.

\bibitem{johnson1975finding}
Donald~B. Johnson.
\newblock Finding all the elementary circuits of a directed graph.
\newblock {\em SIAM J. Comput.}, 4(1):77--84, 1975.

\bibitem{kennel1992determining}
Matthew~B Kennel, Reggie Brown, and Henry~DI Abarbanel.
\newblock Determining embedding dimension for phase-space reconstruction using a geometrical construction.
\newblock {\em Phys. Rev. A}, 45:3403--3411, 1992.

\bibitem{kobayashi2011probability}
Hisashi Kobayashi, Brian~L Mark, and William Turin.
\newblock {\em Probability, random processes, and statistical analysis: Applications to communications, signal processing, queueing theory and mathematical finance}.
\newblock Cambridge University Press, 2011.

\bibitem{PackardCrutchfield}
N.~H. Packard, J.~P. Crutchfield, J.~D. Farmer, and R.~S. Shaw.
\newblock Geometry from a time series.
\newblock {\em Phys. Rev. Lett.}, 45:712--716, 1980.

\bibitem{sauer1991embedology}
Tim Sauer, James~A Yorke, and Martin Casdagli.
\newblock Embedology.
\newblock {\em J. Stat. Phys.}, 65(3):579--616, 1991.

\bibitem{stark}
J.~Stark, D.~S. Broomhead, M.~E. Davies, and J.~Huke.
\newblock Delay embeddings for forced systems. ii. stochastic forcing.
\newblock {\em J. Nonlinear Sci.}, 13:519–577, 2003.

\bibitem{stepaniants2024discovering}
George Stepaniants, Alasdair~D Hastewell, Dominic~J Skinner, Jan~F Totz, and J{\"o}rn Dunkel.
\newblock Discovering dynamics and parameters of nonlinear oscillatory and chaotic systems from partial observations.
\newblock {\em Phys. Rev. Res.}, 6(043062), 2024.

\bibitem{takens}
Floris Takens.
\newblock Detecting strange attractors in turbulence.
\newblock In David Rand and Lai-Sang Young, editors, {\em Dynamical Systems and Turbulence}, volume 898 of {\em Lecture notes in Math.}, pages 366--381. Springer, Berlin, Heidelberg, 1981.

\bibitem{wang2010multi}
Yong Wang, Yuan Jiang, Yi~Wu, and Zhi-Hua Zhou.
\newblock Multi-manifold clustering.
\newblock In Byoung-Tak Zhang and Mehmet~A. Orgun, editors, {\em PRICAI 2010: Trends in Artificial Intelligence}, volume 6230 of {\em Lecture Notes in Computer Science}, pages 280--291. Springer, Berlin, Heidelberg, 2010.

\end{thebibliography}

\end{document}